\documentclass[12pt]{amsart}
\usepackage{amssymb,amsmath,amsthm}
\usepackage{enumerate}

\DeclareMathOperator{\iindex}{index}
\DeclareMathOperator{\sign}{sign}

\DeclareMathOperator{\supp}{supp}
\DeclareMathOperator{\singsupp}{sing\ supp}

\DeclareMathOperator{\Ran}{Ran}
\DeclareMathOperator{\Dom}{Dom}
\DeclareMathOperator{\Ker}{Ker}

\DeclareMathOperator*{\slim}{s-lim}

\renewcommand\Im{\hbox{{\rm Im}}\,}
\renewcommand\Re{\hbox{{\rm Re}}\,}

\newcommand{\abs}[1]{\lvert#1\rvert}

\newcommand{\norm}[1]{\lVert#1\rVert}

\newcommand{\R}{{\mathbb R}}
\newcommand{\N}{{\mathbb N}}

\newcommand{\C}{{\mathbb C}}

\newcommand{\calH}{{\mathcal H}}
\newcommand{\calK}{{\mathcal K}}

\numberwithin{equation}{section}

\theoremstyle{plain}
\newtheorem{theorem}{\bf Theorem}[section]
\newtheorem{lemma}[theorem]{\bf Lemma}
\newtheorem{proposition}[theorem]{\bf Proposition}

\newtheorem*{theorem*}{\bf Theorem}

\theoremstyle{definition}

\theoremstyle{remark}
\newtheorem*{remark*}{\bf Remark}
\newtheorem{remark}[theorem]{\bf Remark}

\renewcommand{\qed}{\vrule height7pt width5pt depth0pt}
\newcommand{\wt}{\widetilde}
\newcommand{\ess}{{\rm ess}}
\newcommand{\ac}{{\rm ac}}

\begin{document}

\title[Discontinuous functions of self-adjoint operators]{Spectral theory of 
discontinuous functions of self-adjoint operators: essential spectrum}
\author{Alexander Pushnitski}
\address{Department of Mathematics\\
King's College London\\ 
Strand, London WC2R  2LS, U.K.}
\curraddr{}
\email{alexander.pushnitski@kcl.ac.uk}
\thanks{}
\date{\today}




\begin{abstract}
Let $H_0$ and $H$ be  self-adjoint operators in a Hilbert space.
In the scattering theory framework, we describe the essential
spectrum of the difference $\varphi(H)-\varphi(H_0)$ for piecewise continuous
functions $\varphi$. This description involves the scattering matrix
for the pair $H$, $H_0$. 
\end{abstract}

\maketitle

\section{Introduction}\label{sec.aa}

Let $H_0$ and $H$ be self-adjoint operators in a Hilbert space $\calH$  and suppose that the difference 
$V=H-H_0$ is a compact operator.  
If $\varphi:\R\to\R$ is a continuous function which tends to zero at infinity then a well known 
simple argument shows that the difference
$$
\varphi(H)-\varphi(H_0)
$$
is a compact operator. 
Moreover, there is a large family of results that assert that if 
the function $\varphi$ is sufficiently ``nice'' and $V$ belongs to 
some Schatten--von Neumann class of compact operators, 
then $\varphi(H)-\varphi(H_0)$ 
also belongs to this class. 
See \cite{Krein,BS1} or the survey \cite{BS2} for early results of this type; 
they   were later made much more precise by V.~V.~Peller,
see \cite{Peller1,Peller2}. See also \cite{Peller3,Peller4} 
for  some recent progress in this area.

In all of the above mentioned results, the function $\varphi$
is assumed to be continuous. If $\varphi$ has discontinuities
on the essential spectrum of $H_0$, then 
the difference $\varphi(H)-\varphi(H_0)$ in general 
fails to be compact even if $V$ is a rank one operator; 
see \cite{Krein,KM}.
In this paper
we study the essential spectrum of $\varphi(H)-\varphi(H_0)$ 
for piecewise continuous functions $\varphi$. 
Some initial results in this direction have been obtained in 
\cite{Push}; we begin by describing these results.

For a Borel set $\Lambda\subset\R$, we denote by $E(\Lambda)$ (resp. $E_0(\Lambda)$)
the spectral projection of $H$ (resp. $H_0$) corresponding to the set $\Lambda$. 
If $\Lambda$ is an interval, say $\Lambda=[a,b)$, we write $E[a,b)$ instead of $E([a,b))$ 
in order to make our formulas more readable. 
In \cite{Push}, under  some 
assumptions typical for smooth scattering theory, it was proven that
for compact $V$ one has
\begin{equation}
\sigma_\ess\bigl(E(-\infty,\lambda)-E_0(-\infty,\lambda)\bigr)
=
\bigl[-\tfrac12\norm{S(\lambda)-I},\tfrac12\norm{S(\lambda)-I}\bigr],
\label{a1}
\end{equation}
where $S(\lambda)$ is the scattering matrix for the pair $H_0$, $H$. 

In this paper, we prove the following generalisation of \eqref{a1}. 
Assume that for some $\lambda\in\R$ the derivatives
\begin{equation}
\frac{d}{d\lambda}\abs{V}^{1/2}E_0(-\infty,\lambda)\abs{V}^{1/2}, 
\qquad
\frac{d}{d\lambda}\abs{V}^{1/2}E(-\infty,\lambda)\abs{V}^{1/2}
\label{a1a}
\end{equation}
exist in the operator norm. Then we prove (see Theorem~\ref{th1})
that the limit 
\begin{equation}
\alpha(\lambda)
=
\lim_{\varepsilon\to+0}\frac{\pi}{2\varepsilon}
\norm{E_0(\lambda-\varepsilon,\lambda+\varepsilon)
VE(\lambda-\varepsilon,\lambda+\varepsilon)}
\label{a1b}
\end{equation}
exists and the identity
\begin{equation}
\sigma_\ess\bigl(E(-\infty,\lambda)-E_0(-\infty,\lambda)\bigr)
=
[-\alpha(\lambda),\alpha(\lambda)]
\label{a1c}
\end{equation}
holds true. 
If the standard assumptions of either trace class or smooth variant of 
scattering theory are fulfilled, we prove (see Section~\ref{sec.a4})
that $\alpha(\lambda)=\frac12\norm{S(\lambda)-I}$. 
Thus, \eqref{a1} becomes a corollary of \eqref{a1c}. 
Using \eqref{a1c}, we obtain the following results: 
\begin{enumerate}[\rm (i)]
\item
Applying \eqref{a1c} in the trace class framework, 
we prove 
(see Section~\ref{sec.a3} for the definition of the core of the absolutely continuous spectrum):
\begin{theorem*}
Let $V$ be a trace class operator. Then for a.e. $\lambda\in\R$ 
the relation \eqref{a1c} holds true and for a.e. $\lambda$ 
in the core 
of the absolutely continuous spectrum of $H_0$, the relation \eqref{a1} holds true. 
\end{theorem*}
This is stated as Theorem~\ref{cr.a3} below.
\item
In Section~\ref{sec.a6}
we describe the essential spectrum of the difference 
$\varphi(H)-\varphi(H_0)$ for piecewise continuous functions $\varphi$. 
\item
In Section~\ref{sec.a5} we give a convenient criterion for 
$E_0(-\infty,\lambda)$, $E(-\infty,\lambda)$ to be a Fredholm pair of 
projections. 
\item
In Sections~\ref{sec.a7}, \ref{sec.a8}, 
we give some applications to the Schr\"odinger operator. 
\end{enumerate}

In the proof of \eqref{a1c} we use the technique of \cite{Push} with 
some minor improvements. In (ii) above, we follow the 
method of proof used by S.~Power in his description \cite{Power} of the 
essential spectrum of Hankel operators with piecewise continuous
symbols.

Finally, we note that a description of the absolutely continuous 
spectrum of the difference
\begin{equation}
E(-\infty,\lambda)-E_0(-\infty,\lambda)
\label{a2}
\end{equation}
is also available in terms of the spectrum of the scattering matrix; 
see \cite{Push,PYaf}.

\section{Main results}\label{sec.a}

\subsection{The definition of the operator  $H$}\label{sec.a2}

Let $H_0$ be a self-adjoint operator in a Hilbert space $\calH$. 
We would like to introduce a self-adjoint perturbation $V$ 
and define the sum $H=H_0+V$. 
Informally speaking, we would like to define $H_0+V$ as a 
quadratic form sum; however, since we do not assume 
$H_0$ or $V$ to be semi-bounded, the language of quadratic
forms is not applicable here. The definition of $H_0+V$ 
requires some care; we follow the approach which goes back at least to  \cite{Kato2}
and was developed in more detail in  \cite[Sections 1.9, 1.10]{Yafaev}.
We assume that $V$ is factorised as $V=G^*JG$, 
where $G$ is an operator from $\calH$ to an auxiliary
Hilbert space $\calK$ and $J$ is an operator
in $\calK$. We assume that 
\begin{equation}
\begin{split}
&J=J^* \text{ is bounded in $\calK$, } 
\\
\Dom \abs{H_0}^{1/2}\subset& \Dom G \text{ and }
G(\abs{H_0}+I)^{-1/2}\text { is compact.}
\end{split}
\label{a3}
\end{equation}
We denote by $(\cdot,\cdot)$ and $\norm{\cdot}$ the 
inner product and the norm in $\calH$ and by 
$(\cdot,\cdot)_\calK$ and $\norm{\cdot}_\calK$
the inner product and the norm in $\calK$.
In applications a  factorisation $V=G^*JG$ with these properties
often arises  naturally from the structure of the problem.
In any case, one can always take
$\calK=\calH$, $G=\abs{V}^{1/2}$ and $J=\sign(V)$.

For $z\in\C\setminus \sigma(H_0)$, we denote
$R_0(z)=(H_0-zI)^{-1}$. Formally, we define the operator $T_0(z)$
(sandwiched resolvent) by  setting
\begin{equation}
T_0(z)=GR_0(z)G^*;
\label{a19}
\end{equation}
more precisely, this means
$$
T_0(z)=G(\abs{H_0}+I)^{-1/2} (\abs{H_0}+I)R_0(z) (G(\abs{H_0}+I)^{-1/2})^*.
$$
By \eqref{a3}, the operator  $T_0(z)$ is compact.
It can be shown 
(see \cite[Sections 1.9,1.10]{Yafaev})
that
under the assumption \eqref{a3} the operator
$I+T_0(z)J$ 
has a bounded inverse for all $z\in\C\setminus\R$ and  
that the operator valued function
\begin{equation}
R(z)=R_0(z)-(G R_0(\overline z))^*J (I+T_0(z)J)^{-1} GR_0(z), 
\quad z\in\C\setminus\R,
\label{a4}
\end{equation}
is a resolvent of a self-adjoint operator; we denote this 
self-adjoint operator by $H$. 
Thus, formula \eqref{a4}, which is usually treated as 
a resolvent identity for $H_0$ and $H=H_0+V$, is 
now accepted as the definition of $H$. 
If $V$ is bounded, then the above defined operator $H$
coincides with the operator sum $H_0+V$.
If $H_0$ is semi-bounded 
from below, then \eqref{a3} means that $V$ is $H_0$-form
compact and then $H$ coincides with the 
quadratic form sum $H_0+V$ 
(in the sense of the KLMN Theorem, see \cite{RS2}). 
In general, we have
\begin{equation}
(f_0, Hf)=(H_0f_0,f)+(JGf_0,Gf)_\calK, 
\quad \forall f_0\in\Dom H_0, \quad \forall f\in\Dom H.
\label{a4a}
\end{equation}
Finally,   it is not difficult to check that by \eqref{a3} and \eqref{a4},
the resolvent $R(z)$ can be written as
\begin{equation}
R(z)=(\abs{H_0}+I)^{-1/2}\mathcal B(z)(\abs{H_0}+I)^{-1/2}
\label{a4b}
\end{equation}
with a bounded operator $\mathcal B(z)$. 
In particular, this implies that 
\begin{equation}
\text{the operator $GR(z)$ is well defined and
compact for any $z\in\C\setminus\R$. }
\label{a20}
\end{equation}

\subsection{Main result}\label{sec.a2a}
Let us fix a ``reference point" $\nu\in\R$ and for 
$\lambda>\nu$ denote 
\begin{equation}
\begin{split}
F_0(\lambda)&=GE_0[\nu,\lambda)\bigl(GE_0[\nu,\lambda)\bigr)^*,
\\
F(\lambda)&=GE[\nu,\lambda)\bigl(GE[\nu,\lambda)\bigr)^*.
\end{split}
\label{a5}
\end{equation}
Note that by \eqref{a3} and \eqref{a20}, the operators $GE_0[\nu,\lambda)$ and 
$GE[\nu,\lambda)$ 
are well defined and compact. 
For $\nu<\lambda_1<\lambda_2$, we have
\begin{equation}
F_0(\lambda_2)-F_0(\lambda_1)
=
GE_0[\lambda_1,\lambda_2)\bigl(GE_0[\lambda_1,\lambda_2)\bigr)^*
\label{a5a}
\end{equation}
and a similar identity holds true for $F(\lambda)$.
In what follows, we discuss the derivatives
\begin{equation}
F_0'(\lambda)=\frac{d}{d\lambda}F_0(\lambda), 
\quad
F'(\lambda)=\frac{d}{d\lambda}F(\lambda)
\label{a6}
\end{equation}
understood in the operator norm sense. 
By \eqref{a5a}, 
it is clear that neither the existence nor the values of these derivatives 
depend on the choice of the reference point $\nu$. 
In fact, if $H_0$ is semi-bounded from below, then we can take $\nu=-\infty$.
It is also clear that if these derivatives exist
in the operator norm, then $F_0'(\lambda)\geq0$ and $F'(\lambda)\geq0$ 
in the quadratic form sense. 
\begin{theorem}\label{th1}
Assume \eqref{a3} and suppose that for some $\lambda>\nu$, 
the derivatives $F_0'(\lambda)$, $F'(\lambda)$ exist in the operator norm. 
Then the limit
\begin{equation}
\alpha(\lambda)\overset{\rm def}{=}
\lim_{\varepsilon\to+0}\frac{\pi}{2\varepsilon}
\norm{(GE_0(\lambda-\varepsilon,\lambda+\varepsilon))^*J
GE(\lambda-\varepsilon,\lambda+\varepsilon)}
\label{a7a}
\end{equation}
exists and the identity 
\begin{equation}
\sigma_\ess\bigl(E(-\infty,\lambda)-E_0(-\infty,\lambda)\bigr)
=[-\alpha(\lambda),\alpha(\lambda)]
\label{a8}
\end{equation}
holds true. One also has 
\begin{equation}
\alpha(\lambda)=\pi\norm{F_0'(\lambda)^{1/2}JF'(\lambda)^{1/2}}.
\label{a7}
\end{equation}
\end{theorem}
The proof is given in Section~\ref{sec.b}. 
\begin{remark*}
\begin{enumerate}[\rm 1.]
\item
It is straightforward to see that $\sigma(E(-\infty,\lambda)-E_0(-\infty,\lambda))\subset[-1,1]$. 
Thus, Theorem~\ref{th1} implies, in particular, that $\alpha(\lambda)\leq1$. 
\item
If $\lambda\notin\sigma(H_0)$, then $F_0'(\lambda)=0$, and we obtain
that the difference of the spectral projections in \eqref{a8} is compact. 
This is not difficult to prove directly (see Remark~\ref{rm.b1}). 
\item
If the operator $VR_0(i)$ is bounded, then it is obvious that \eqref{a7a} 
can be rewritten as \eqref{a1b}. 
\end{enumerate}
\end{remark*}

In what follows we prove that under the standard assumptions
of either trace class or smooth version of scattering theory, one has 
\begin{equation}
\alpha(\lambda)=\norm{S(\lambda)-I}/2,
\label{a9}
\end{equation}
where $S(\lambda)$ is the scattering matrix for the pair $H_0$, $H$. 
Thus, the verification of \eqref{a1} splits into two parts: \eqref{a8} and \eqref{a9}.
The statement \eqref{a8} is more general than \eqref{a9}. 
Indeed, in order to prove \eqref{a9},  one has to ensure that 
the scattering matrix $S(\lambda)$ is well defined;  this requires some
assumptions  stronger than those of Theorem~\ref{th1}, see
Sections~\ref{sec.a3} and \ref{sec.a4}.

\subsection{The scattering matrix}\label{sec.a3}
Here, following \cite{Yafaev}, we recall the definition of the scattering 
matrix in abstract scattering theory. 
This requires some rather lengthy preliminaries.
First we need to recall the definition of the core of the
absolutely continuous (a.c.) spectrum of $H_0$. 
Let $E_0^{(\ac)}(\cdot)$ (resp. $E^{(\ac)}(\cdot)$)
be the a.c. part of the spectral measure of $H_0$ (resp. $H$)
and let $\sigma_\ac(H_0)$ be the a.c. spectrum of $H_0$ defined 
as usual as the minimal \emph{closed} set such that 
$E_0^{(\ac)}(\R\setminus\sigma_\ac(H_0))=0$. 

The set $\sigma_\ac(H_0)$ is ``too large'' for 
general scattering theory considerations. 
Indeed, it is not difficult to construct examples when $\sigma_\ac(H_0)$ 
contains a closed set $A$ of a \emph{positive Lebesgue measure} 
such that $E_0^{(\ac)}(A)=0$ (consider $E_0^{(\ac)}$ being supported 
on the intervals $(a_n-2^{-n}, a_n+2^{-n})$, 
where $a_1,a_2,\dots$ is a dense sequence in $\R$). 
Thus, it is convenient to use the notion of the \emph{core of the a.c. spectrum
of $H_0$}, denoted by $\hat \sigma_\ac(H_0)$ and defined as a Borel set
such that:

(i) $\hat \sigma_\ac(H_0)$ is a Borel support of $E_0^{(\ac)}$, i.e. 
$E_0^{(\ac)}(\R\setminus\hat \sigma_\ac(H_0))=0$; 

(ii) if $A$ is any other Borel support of $E_0^{(\ac)}$, then 
the set $\hat \sigma_\ac(H_0)\setminus A$ has a zero Lebesgue measure. 

The set $\hat \sigma_\ac(H_0)$ is not unique but 
is defined up to a set of a zero Lebesgue measure. 

Suppose that for some interval $\Delta\subset \R$, the (local)
wave operators 
$$
W_\pm=W_\pm(H_0,H;\Delta)
=
\slim_{t\to\pm\infty}
e^{itH}e^{-itH_0}E_0^{(\ac)}(\Delta)
$$
exist and $\Ran W_+(H_0,H;\Delta)=\Ran W_-(H_0,H;\Delta)$. 
Then the (local) scattering operator $\mathbf S=W_+^*W_-$ 
is unitary in $\Ran E_0^{(\ac)}(\Delta)$ and commutes
with $H_0$. Consider the direct integral decomposition
\begin{equation}
\Ran E_0^{(\ac)}(\Delta)=
\int_{\hat\sigma_{\ac}(H_0)\cap \Delta}^\oplus \mathfrak h(\lambda)d\lambda
\label{a10}
\end{equation}
which diagonalises $H_0$. Since $\mathbf S$ commutes with $H_0$, 
the decomposition \eqref{a10} represents $\mathbf S$ as the operator
of multiplication by the operator valued function 
$S(\lambda):\mathfrak h(\lambda)\to\mathfrak h(\lambda)$. 
The unitary operator $S(\lambda)$ is called the scattering matrix. 
With this definition, $S(\lambda)$ is defined for a.e. 
$\lambda\in\hat\sigma_\ac(H_0)$. In abstract scattering theory, it does not make
sense to speak of $S(\lambda)$ at an individual point 
$\lambda\in\hat\sigma_\ac(H_0)$, since even the set 
$\hat\sigma_\ac(H_0)$  is defined only up to addition or
subtraction of sets of zero Lebesgue measure. 
Also, in general there is no 
distinguished choice of the direct integral decomposition
\eqref{a10}; any unitary transformation in the fiber spaces $\mathfrak h(\lambda)$ 
yields another suitable decomposition. Thus, the scattering 
matrix is, in general, defined only up to a unitary equivalence. 

The above discussion refers only to the ``abstract'' version 
of the mathematical scattering theory. 
In concrete problems, there is often a natural distinguished 
choice of the core $\hat\sigma_\ac(H_0)$ and of the
direct integral decomposition \eqref{a10}. 
This usually allows one to consider $S(\lambda)$ as 
an operator defined
for \emph{all} (rather than for a.e.) $\lambda\in\hat\sigma_\ac(H_0)$. 

In what follows we set
\begin{equation}
S(\lambda)=I \quad \text{for $\lambda\in\R\setminus\hat\sigma_\ac(H_0)$;}
\label{a23}
\end{equation}
thus, $S(\lambda)$ is now defined for a.e. $\lambda\in\R$. 
This will make the statements below more succinct.

\subsection{The scattering matrix and $\alpha(\lambda)$}\label{sec.a4}
Similarly to the definition \eqref{a19} of $T_0(z)$, let us formally define 
$T(z)=GR(z)G^*$. More precisely, using \eqref{a4b}, we set
$$
T(z)=G(\abs{H_0}+I)^{-1/2} \mathcal B(z) (G(\abs{H_0}+I)^{-1/2})^*.
$$
By \eqref{a3}, the operator $T(z)$ is compact.  
From the resolvent identity \eqref{a4} it follows that 
\begin{equation}
T(z)=T_0(z)-T_0(z)J(I+T_0(z)J)^{-1}T_0(z)
=(I+T_0(z)J)^{-1}T_0(z).
\label{a21}
\end{equation} 
First let us consider the framework of smooth perturbations. 
Suppose that for some bounded open interval
$\Delta\subset\R$, 
\begin{equation}
\begin{split}
&\text{$T_0(z)$ and $T(z)$ are uniformly continuous in the operator norm}
\\
&\text{in the rectangle $\Re z\in\Delta$,\   \   $\Im z\in(0,1)$.}
\end{split}
\label{a11}
\end{equation}
Of course, from here it trivially follows that the limits
$T_0(\lambda+i0)$, $T(\lambda+i0)$ exist in the operator norm 
for all $\lambda\in\Delta$. 
Under the assumption \eqref{a11} the operator $G$ is locally $H_0$-smooth and $H$-smooth on 
$\Delta$, and therefore the local wave operators $W_\pm(H_0,H;\Delta)$ 
exist and are complete (see e.g. \cite{Yafaev} for the details). 
The scattering matrix $S(\lambda)$ is defined for a.e. 
$\lambda\in\hat \sigma_\ac(H_0)\cap\Delta$. 
\begin{theorem}\label{cr.a2}
Assume \eqref{a3} and \eqref{a11}. 
Then for all $\lambda\in\Delta$, the derivatives $F_0'(\lambda)$ and $F'(\lambda)$ 
exist in the operator norm and so \eqref{a8} holds true. 
For a.e. $\lambda\in\Delta$, the identities \eqref{a9} and \eqref{a1}
hold true. 
\end{theorem}
The proof is given in Section~\ref{sec.c2}. 
In \cite{Push}, formula \eqref{a1} was proven under the additional 
assumptions of the compactness of $G$ (which is a stronger assumption
than \eqref{a3})
and the H\"older continuity of 
$F_0'(\lambda)$ and $F'(\lambda)$. 

Next, consider the trace class scheme. 
Let $\mathbf S_2$ be the Hilbert-Schmidt class. 
Suppose that $H=H_0+V$, where $V=V^*$ is a trace
class operator. Then we can factorise $V=GJG^*$
with $G=\abs{V}^{1/2}\in\mathbf S_2$ and $J=\sign(V)$. 
It is well known that under these assumptions, 
the derivatives $F_0'(\lambda)$ and $F'(\lambda)$ 
exist in the operator norm 
for a.e. $\lambda\in\R$ (see e.g. \cite[Section~6.1]{Yafaev}).
We have
\begin{theorem}\label{cr.a3}
Let $H=H_0+V$, where $V$ is a trace class operator. 
Set $G=\abs{V}^{1/2}$. 
Then for a.e. $\lambda\in\R$, the derivatives 
$F_0'(\lambda)$, $F'(\lambda)$ exist and 
\eqref{a8}, \eqref{a9} and \eqref{a1} hold true. 
\end{theorem}

Alternatively, we have the following statement 
more suitable for applications to differential operators:

\begin{theorem}\label{cr.a4}
Let $H_0$ be semi-bounded from below; assume that 
\eqref{a3} holds true and also, for some $m>0$, 
\begin{equation}
G(\abs{H_0}+I)^{-m}\in \mathbf S_2.
\label{a13}
\end{equation}
Then the conclusion of Theorem~\ref{cr.a3} holds true. 
\end{theorem}
The proofs of Theorems~\ref{cr.a3} and \ref{cr.a4} are given in Section~\ref{sec.c2}. 

\begin{remark*}
\begin{enumerate}[1.]
\item
The existence and completeness of the 
wave operators under the assumptions of Theorems~\ref{cr.a3} and \ref{cr.a4}
is well known; 
see e.g. \cite[Section~4.5 and Section~6.4]{Yafaev}.
\item
According to our convention \eqref{a23}, we have
$$
\norm{S(\lambda)-I}=0 \quad \text{for $\lambda\in\R\setminus\hat\sigma_\ac(H_0)$.}
$$
Thus, 
Theorems~\ref{cr.a2}--\ref{cr.a4} in particular, include the statement that
for a.e. $\lambda\in\R\setminus\hat\sigma_\ac(H_0)$,
the difference of the spectral projections \eqref{a2} is compact. 
\end{enumerate}
\end{remark*}

\subsection{Piecewise continuous functions $\varphi$ }\label{sec.a6}

Let us consider the spectrum of $\varphi(H)-\varphi(H_0)$
for piecewise continuous functions $\varphi$. 
It is natural to consider complex-valued functions $\varphi$; 
in this case $\varphi(H)-\varphi(H_0)$ is non-selfadjoint. 
For a bounded  operator $M$, we denote by 
$\sigma_\ess(M)$ the compact set of all $z\in\C$ such that
the operator $M-z I$ is not Fredholm. 
The reader should be warned that there
are several non-equivalent definitions of the essential spectrum
of a non-selfadjoint operator
in the literature; see e.g. \cite[Sections~1.4 and 9.1]{EE}
for a comprehensive discussion. However, 
as we shall see, the essential 
spectrum of $\varphi(H)-\varphi(H_0)$ has an empty interior and  
a connected complement in $\C$, and so 
in our case most of these definitions coincide.

A function $\varphi:\R\to\C$
is called piecewise continuous if for any $\lambda\in\R$ the limits
$\varphi(\lambda\pm0)=\lim_{\varepsilon\to+0}\varphi(\lambda\pm\varepsilon)$ exist. 
We denote by  $PC_0(\R)$ (resp. $C_0(\R)$) the set of all 
piecewise continuous (resp. continuous) 
functions $\varphi:\R\to\C$ such that $\lim_{\abs{x}\to\infty}\varphi(x)=0$. 
For $\varphi\in PC_0(\R)$ we denote 
$$
\varkappa_\lambda(\varphi)=\varphi(\lambda+0)-\varphi(\lambda-0),
\quad
\singsupp \varphi=
\overline{\{\lambda\in\R\mid \varkappa_\lambda(\varphi)\not=0\} }.
$$
It is easy to see that for any $\varepsilon>0$, the set 
$\{\lambda\in\R\mid \abs{\varkappa_\lambda(\varphi)}>\varepsilon\}$
is finite. 
For $z_1,z_2\in\C$, we denote by $[z_1,z_2]$ the closed 
interval of the straight line in $\C$ 
that joins  $z_1$ and $z_2$.

\begin{theorem}\label{th10}
Assume \eqref{a3} and let \eqref{a11} hold true for some open bounded interval 
$\Delta\subset\R$. Let $\varphi\in PC_0(\R)$ be a  
function with $\singsupp \varphi\subset\Delta$. Then
we have
\begin{equation}
\sigma_\ess(\varphi(H)-\varphi(H_0))=
\cup_{\lambda\in\Delta}
[-\alpha(\lambda)\varkappa_\lambda(\varphi),
\alpha(\lambda)\varkappa_\lambda(\varphi)],
\label{a16}
\end{equation}
where $\alpha(\lambda)$ is defined by \eqref{a7a}.
In particular, if $\varphi$ is real valued, then 
$$
\sigma_\ess(\varphi(H)-\varphi(H_0))=[-a,a],
\quad
a=\sup_{\lambda\in\Delta}\abs{\alpha(\lambda)\varkappa_\lambda(\varphi)}.
$$
\end{theorem}
The proof is given in Section~\ref{sec.d}.

\subsection{The Fredholm property}\label{sec.a5}
A pair of orthogonal projections $P$, $Q$ in a Hilbert space
is called Fredholm, if (see e.g. \cite{ASS})
\begin{equation}
\pm1\notin\sigma_\ess(P-Q). 
\label{a13a}
\end{equation}
If $P,Q$ is a Fredholm pair, one defines the index of $P,Q$ by 
$$
\iindex(P,Q)=\dim\Ker(P-Q-I)-\dim\Ker(P-Q+I).
$$
In a forthcoming publication \cite{Push3}, we study the index of 
the pair 
\begin{equation}
E(-\infty,\lambda),\quad E_0(-\infty,\lambda).
\label{a14}
\end{equation}
In connection with this (and perhaps otherwise)
it is interesting to know whether the pair  \eqref{a14}
is Fredholm. Under the assumptions of Theorem~\ref{th1}, the 
question reduces to deciding whether $\alpha(\lambda)<1$ or $\alpha(\lambda)=1$. 
If \eqref{a9} holds true, then, clearly, the pair \eqref{a14} is 
Fredholm if and only if $-1$ is not an eigenvalue 
of the scattering matrix $S(\lambda)$. 
Below we give a convenient criterion for this in terms of 
the operators $T_0$, $T$.

For a bounded operator $M$, we denote $\Re M=(M+M^*)/2$, 
$\Im M=(M-M^*)/2i$. 
If the limits $T_0(\lambda+i0)$, $T(\lambda+i0)$ exist, we denote
\begin{equation}
A_0(\lambda)=\Re T_0(\lambda+i0), 
\quad
A(\lambda)=\Re T(\lambda+i0).
\label{a15}
\end{equation}
\begin{theorem}\label{th5}
Assume \eqref{a3}. Suppose that for some $\lambda\in\R$, the limits
$T_0(\lambda+i0)$, $T(\lambda+i0)$ and the derivatives
$F_0'(\lambda)$, $F'(\lambda)$ exist in the operator norm. 
Then the following statements are equivalent: 

(i) the pair \eqref{a14} is Fredholm;

(ii) $\Ker(I+A_0(\lambda)J)=\{0\}$; 

(iii) $\Ker(I-A(\lambda)J)=\{0\}$.
\end{theorem}
The proof is given in Section~\ref{sec.c}. 
Theorem~\ref{th5} can be applied to either smooth or trace class
framework. 
In applications, one can often obtain some information about the spectrum
of $A_0(\lambda)$ or $A(\lambda)$; for example, one can sometimes
ensure that the norm of $A_0(\lambda)$ is small. 
By Theorem~\ref{th5}, this can be used to ensure that the 
pair \eqref{a14} is Fredholm.

\begin{remark*}
\begin{enumerate}[1.]
\item
Since $\dim\Ker(I+XY)=\dim\Ker(I+YX)$ for any bounded operators $X$, $Y$, 
we can equivalently restate (ii), (iii) as 

(iia) $\Ker(I+JA_0(\lambda))=\{0\}$; 

(iiia) $\Ker(I-JA(\lambda))=\{0\}$.

\item
If the operator $J$ has a bounded inverse, 
we can equivalently restate (ii), (iii) in a more symmetric form as 

(iib) $\Ker(J^{-1}+A_0(\lambda))=\{0\}$; 

(iiib) $\Ker(J^{-1}-A(\lambda))=\{0\}$.

\end{enumerate}
\end{remark*}

\subsection{Schr\"odinger operator: smooth framework}\label{sec.a7}
Let $H_0=-\Delta$ in $\mathcal H=L^2(\R^d)$, $d\geq 1$, and let
$H=H_0+V$, where $V$ is the operator of multiplication by a function
$V:\R^d\to\R$ which is assumed to satisfy
\begin{equation}
\abs{V(x)}\leq C(1+\abs{x})^{-\rho}, \quad \rho>1.
\label{a18}
\end{equation}
Let $\mathcal K=\mathcal H$, $G=\abs{V}^{1/2}$, $J=\sign V$. 
Under the assumption \eqref{a18}, 
the hypotheses \eqref{a3} and \eqref{a11} hold true with any 
$\Delta=(c_1,c_2)$, $0<c_1<c_2<\infty$, 
see e.g. \cite[Theorem~XIII.33]{RS4}. 
It is also easy to see 
that the derivatives $F_0'(\lambda)$, $F'(\lambda)$ exist in the operator 
norm for all $\lambda>0$. 
Thus, for any $\lambda>0$, the conclusions of Theorems~\ref{cr.a2}
and \ref{th5} hold true. In \cite{Push}, 
formula \eqref{a1} was proven for $H_0$ and $H$ as above 
only for $d\leq 3$. 
We also see that the conclusion of Theorem~\ref{th10} holds true
for any  $\varphi\in PC_0(\R)$ which is continuous in an open neighbourhood
of zero.

In this example there is a well known natural choice of the core
$\hat \sigma_\ac(H_0)=(0,\infty)$ and of the direct integral 
decomposition \eqref{a10} with $\mathfrak h(\lambda)=L^2(\mathbb S^{d-1})$. 
Moreover, in this case the scattering matrix 
$S(\lambda): L^2(\mathbb S^{d-1})\to L^2(\mathbb S^{d-1})$ 
is continuous in $\lambda>0$.
Thus, in this case the statement \eqref{a1} holds true for all $\lambda>0$.

\subsection{Schr\"odinger operator: trace class framework}\label{sec.a8}
Let $H_0=-\Delta+U$ in $L^2(\R^d)$, $d\geq1$, where $U$ is the operator
of multiplication by a real valued bounded function. 
Next, let $H=H_0+V$, where $V$ is the operator of multiplication
by a real valued function $V\in L^1(\R^d)$ such that $V$ 
is $(-\Delta)$-form compact. Then $V$ is also $H_0$-form compact
and $H=H_0+V$ is well defined as a form sum. It is well known 
(see e.g. \cite[Theorem~B.9.1]{Simon}) that under the above assumptions, 
\eqref{a13} holds true with $G=\abs{V}^{1/2}$ for $m>d/4$. 
Thus, the conclusions of Theorem~\ref{cr.a4} hold true. 

The assumptions on $H_0$ in this example can be considerably 
relaxed by allowing $U$ to have
local singularities, by including a background magnetic field, etc.
Note that in this example we have no information on the 
a.c. spectrum of $H_0$.

\section{Proof of Theorem~\ref{th1}}\label{sec.b}

We follow the method of \cite{Push} with some minor technical 
improvements. 
In order to simplify our notation, we assume $\lambda=0$
and denote $\R_+=(0, \infty)$, $\R_-=(-\infty,0)$. 

\subsection{The proof of \eqref{a7}}\label{sec.b0}
Let us prove that if the derivatives  $F_0'(0)$ and $F'(0)$ exist 
in the operator norm, then the limit \eqref{a7a} also exists 
and the identity \eqref{a7} holds true. 
Let us start from the r.h.s. of \eqref{a7}. 
Denoting $\delta_\varepsilon=(-\varepsilon,\varepsilon)$
and using the identities $\norm{X}^2=\norm{XX^*}=\norm{X^*X}$, 
we get
\begin{align*}
\norm{F_0'(0)^{1/2}JF'(0)^{1/2}}^2
&=
\norm{F_0'(0)^{1/2}JF'(0)JF_0'(0)^{1/2}}
\\
&=
\lim_{\varepsilon\to+0}\frac1\varepsilon
\norm{F_0'(0)^{1/2}JGE(\delta_\varepsilon)(GE(\delta_\varepsilon))^*JF_0'(0)^{1/2}}
\\
&=
\lim_{\varepsilon\to+0}\frac1\varepsilon
\norm{(GE(\delta_\varepsilon))^*JF_0'(0)JGE(\delta_\varepsilon)}
\\
&=
\lim_{\varepsilon\to+0}\frac1{\varepsilon^2}
\norm{(GE(\delta_\varepsilon))^*JGE_0(\delta_\varepsilon)(GE_0(\delta_\varepsilon))^*JGE(\delta_\varepsilon)}
\\
&=
\lim_{\varepsilon\to+0}\frac1{\varepsilon^2}
\norm{(GE(\delta_\varepsilon))^*JGE_0(\delta_\varepsilon)}^2,
\end{align*}
as required. 

In the rest of this section, we prove that if the derivatives 
$F_0'(0)$ and $F'(0)$ exist 
in the operator norm, then the identity
\begin{equation}
\sigma_\ess\bigl(E(\R_-)-E_0(\R_-)\bigr)=[-\alpha(0),\alpha(0)]
\label{b1}
\end{equation}
holds true with 
$\alpha(0)=\pi\norm{F_0'(0)^{1/2}JF'(0)^{1/2}}$.

\subsection{The kernels of $H_0$ and $H$ }\label{sec.b1}

\begin{lemma}\label{lma.b1}
Assume \eqref{a3} and suppose that the derivatives $F_0'(0)$ and $F'(0)$ 
exist in the operator norm. Then $\Ker H_0=\Ker H$. 
\end{lemma}
\begin{proof}
1. By our assumptions, $GE_0(\{0\})=0$ (otherwise $F_0'(0)$ cannot exist). 
Suppose $\psi\in\Ker H_0$; then $G\psi=0$ and the resolvent identity 
\eqref{a4} yields 
$$
R(z)\psi=R_0(z)\psi=-\frac1z\psi.
$$
Thus, $H\psi=0$. This proves that $\Ker H_0\subset \Ker H$. 

2. From \eqref{a4} it is not difficult to obtain the ``usual''
resolvent identity (see e.g. \cite[Section~1.10]{Yafaev}):
\begin{equation}
R(z)=R_0(z)-(GR_0(\overline z))^*JGR(z).
\label{b2}
\end{equation}
Now let $\psi\in\Ker H$. 
As above, $GE(\{0\})=0$, and so  from \eqref{b2} one obtains
$$
R_0(z)\psi=R(z)\psi=-\frac1z\psi.
$$
Thus, $H_0\psi=0$ and so $\Ker H\subset \Ker H_0$. 
\end{proof}

\subsection{Reduction to the products of spectral projections}\label{sec.b2}
Let us denote 
$$
D=E(\R_-)-E_0(\R_-), \quad
\calH_+=\Ker(D-I), \quad \calH_-=\Ker(D+I), 
\quad \calH_0=(\calH_+\oplus\calH_-)^\perp.
$$
It is well known (see e.g. \cite{Halmos} or \cite{ASS}) that 
\begin{equation}
D|_{\calH_0} \text{ is unitarily equivalent to } (-D)|_{\calH_0}. 
\label{b3}
\end{equation}
Therefore, the spectral analysis of $D$ reduces to the spectral 
analysis of $D^2$ and to the evaluation of the dimensions 
of $\calH_+$ and $\calH_-$. 

Next, using Lemma~\ref{lma.b1}, by a simple algebra 
we obtain the identity 
\begin{equation}
D^2=E_0(\R_-)E(\R_+)E_0(\R_-)+E_0(\R_+)E(\R_-)E_0(\R_+), 
\label{b4}
\end{equation}
where the r.h.s. provides a block-diagonal decomposition 
of $D^2$ with respect to the direct sum 
$$
\calH=\Ran E_0(\R_-)\oplus\Ran E_0(\{0\})\oplus \Ran E_0(\R_+).
$$
Thus, the spectral analysis of $D^2$ reduces to the spectral 
analysis of the two terms in the r.h.s. of \eqref{b4}. 
In Sections~\ref{sec.b3}--\ref{sec.b6} we prove
\begin{lemma}\label{lma.b2}
Assume \eqref{a3}. Then the differences
\begin{gather}
E_0(\R_+)E(\R_-)E_0(\R_+)-E_0(0,1)E(-1,0)E_0(0,1)
\label{b6}
\\
E_0(\R_-)E(\R_+)E_0(\R_-)-E_0(-1,0)E(0,1)E_0(-1,0)
\label{b5}
\end{gather}
are compact operators. 
\end{lemma}
\begin{theorem}\label{thm.b3}
Assume \eqref{a3} and suppose that the derivatives
$F_0'(0)$ and $F'(0)$ exist in the operator norm. 
Then 
\begin{gather}
\sigma_\ess(E_0(0,1)E(-1,0)E_0(0,1))=[0,\alpha(0)^2],
\label{b8}
\\
\sigma_\ess(E_0(-1,0)E(0,1)E_0(-1,0))=[0,\alpha(0)^2],
\label{b7}
\end{gather}
where $\alpha(0)$ is given by 
$\alpha(0)=\pi\norm{F_0'(0)^{1/2}JF'(0)^{1/2}}$.
In particular, $\alpha(0)\leq 1$. 
\end{theorem}
With these two statements, it is easy to provide 
\begin{proof}[Proof of Theorem~\ref{th1}]
Combining 
Lemma~\ref{lma.b2}, Theorem~\ref{thm.b3}
and Weyl's theorem on the invariance of the essential 
spectrum under  compact perturbations, we obtain
\begin{equation}
\sigma_\ess(E_0(\R_-)E(\R_+)E_0(\R_-))
=\sigma_\ess(E_0(\R_+)E(\R_-)E_0(\R_+))
=[0,\alpha(0)^2].
\label{b9}
\end{equation}
By \eqref{b4}, it follows that
\begin{equation}
\sigma_\ess(D^2)=[0,\alpha(0)^2].
\label{b23}
\end{equation}
Suppose first that $\alpha(0)=1$. 
Then from \eqref{b23} and \eqref{b3} we obtain
$\sigma_\ess(D)=[-1,1]$, as required. 
Next, suppose $\alpha(0)<1$. 
Then from \eqref{b23} it follows that the dimensions
of $\calH_-$ and $\calH_+$ are finite, and therefore
$$
\sigma_\ess(D)=\sigma_\ess(D|_{\calH_0})
\text{ and }
\sigma_\ess(D^2)=\sigma_\ess((D|_{\calH_0})^2).
$$
Recalling \eqref{b3}, 
we obtain
$$
\sigma_\ess(D|_{\calH_0})=[-\alpha(0),\alpha(0)],
$$
and
\eqref{b1} follows. 
\end{proof}

\subsection{Proof of Lemma~\ref{lma.b2}}\label{sec.b3}

\begin{lemma}\label{lma.b4}
Assume \eqref{a3}. Let $\varphi\in C(\R)$ be a function such that
the limits $\lim_{x\to\pm\infty}\varphi(x)$ exist. 
Then the difference $\varphi(H)-\varphi(H_0)$ is compact. 
\end{lemma}

\begin{proof}

As is well known (and can easily  be deduced from the compactness of $R(z)-R_{0}(z)$ for $\Im z\neq 0$),
the operator $\varphi(H)-\varphi(H_0)$ 
is compact for any function $\varphi\in C_0(\R)$. 
Therefore, it suffices to prove that $\varphi(H)-\varphi(H_0)$
is compact for at least one function $\varphi\in C(\R)$ 
such that $\lim_{x\to\infty}\varphi(x)\not=\lim_{x\to-\infty}\varphi(x)$
and both limits exist. 
The latter fact is provided by  \cite[Theorem~7.3]{Push2} where 
it is proven that if \eqref{a3}
holds true then the difference $\tan^{-1}(H)-\tan^{-1}(H_0)$ is compact. 
\end{proof}

\begin{remark}\label{rm.b1}
Let $\mu\in\R\setminus(\sigma(H_0)\cup \sigma(H))$. 
Then 
\begin{equation}
E(-\infty, \mu)-E_0(-\infty, \mu)
=
\varphi(H)-\varphi(H_0)
\label{b10}
\end{equation}
for an appropriately chosen continuous function $\varphi$ with 
$\varphi(x)=1$ for  $x\in\sigma(H)\cup\sigma(H_0)$, 
$x<\mu$ and $\varphi(x)=0$ for  $x\in\sigma(H)\cup\sigma(H_0)$, 
$x>\mu$. 
It follows that the difference \eqref{b10} is compact. 
\end{remark}

\begin{proof}[Proof of Lemma~\ref{lma.b2}]
1. Let $\varphi_1\in C(\R)$ be such that $\varphi_1(x)=1$ for $x\leq-1$ 
and $\varphi_1(x)=0$ for $x\geq 0$. Then 
\begin{equation}
E(-\infty, -1)E_0(\R_+)
=
E(-\infty,-1)(\varphi_1(H)-\varphi_1(H_0))E_0(\R_+)
\label{b11}
\end{equation}
and so by Lemma~\ref{lma.b4} the r.h.s. is compact. 

2. Let $\varphi_2\in C(\R)$ be such that $\varphi_2(x)=1$ for $x\geq 1$ and 
$\varphi_2(x)=0$ for $x\leq 0$. Then 
\begin{equation}
E_0(1,\infty)E(\R_-)
=
E_0(1,\infty)(\varphi_2(H_0)-\varphi_2(H))E(\R_-),
\label{b12}
\end{equation}
and so by Lemma~\ref{lma.b4} the r.h.s. is compact. 

3. From the compactness of the l.h.s. of \eqref{b11} and \eqref{b12}, 
the compactness of the difference \eqref{b6} follows by some simple algebra. 
Compactness of \eqref{b5} is proven in the same way. 
\end{proof}

\subsection{Hankel operators}\label{sec.b4}
In order to prove Theorem~\ref{thm.b3}, 
we need some basic facts concerning operator valued 
Hankel  integral operators.
Suppose that for each $t>0$, a bounded self-adjoint operator $K(t)$ in $\calK$ is given. 
Suppose that $K(t)$ is continuous in $t>0$ in the operator norm.
Define a Hankel  integral operator $K$ in $L^2(\R_+,\calK)$ by
\begin{equation}
(Kf,g)_{L^2(\R_+,\calK)}=\int_0^\infty \int_0^\infty 
(K(t+s)f(t),g(s))_{\calK}\, dt\,ds,
\label{b13}
\end{equation}
when $f,g\in L^2(\R_+,\calK)$ are functions with 
compact support in $\R_+$. 
The statement below is a straightforward generalisation of
\cite[Proposition~1.1]{Howland3} 
to the operator valued case. 
\begin{proposition}\label{lma.b5}

(i) Suppose $\norm{K(t)}\leq C/t$ for all $t>0$.
Then the operator $K$ is bounded and $\norm{K}\leq \pi C$. 

(ii) Suppose $K(t)$ is compact for all $t$ and $\norm{K(t)}=o(1/t)$ as $t\to+0$ and 
as $t\to+\infty$. Then $K$ is compact.
\end{proposition}
\begin{proof}
Since the Carleman operator on $L^2(\R_+)$ with the kernel $(t+s)^{-1}$ 
is bounded with the norm $\pi$, we have 
$$
\abs{(Kf,g)_{L^2(\R_+,\calK)}}\leq C \int_0^\infty \int_0^\infty 
\frac{\norm{f(t)}_{\calK}\norm{g(s)}_{\calK}}{t+s}dt\,ds
\leq
\pi C \norm{f}_{L^2(\R_+,\calK)} \norm{g}_{L^2(\R_+,\calK)},
$$
which proves (i). To prove (ii), we need to approximate $K$ by compact
operators. Let $K_n(t)=K(t)\chi_n(t)$, where $\chi_n$ is the characteristic
function of the interval $(1/n,n)$ and let $K_n$ be the corresponding
operator in $L^2(\R_+,\calK)$. 
By (i), $\norm{K-K_n}_{L^2(\R_+,\calK)}\to0$ as $n\to\infty$.
Thus, it remains to show that each $K_n$ is compact.

For each $n$, the Hankel type integral operator with the kernel 
$\chi_n(t+s)/(t+s)$ in $L^2(\R_+)$ is compact (in fact, Hilbert-Schmidt). 
It follows that $K_n$ is compact if $K(t)$ is independent of $t$. 
Now the result follows from the fact that $K(t)$ can be uniformly approximated
by piecewise constant functions on the interval $(1/n,n)$.   
\end{proof}

Important model operators in our construction below are 
the Hankel  integral operators in 
$L^2(\R_+, \calK)$ of the type \eqref{b13} with $K(t)$ given by 
\begin{equation}
\frac{1-e^{-t}}{t}F_0'(0)
\quad\text{ and }\quad
\frac{1-e^{-t}}{t}F'(0). 
\label{b14}
\end{equation}
For this reason, we need to discuss the integral Hankel operator
$\Gamma$ in $L^2(\R_+)$ with the integral kernel 
$\Gamma(t,s)=\frac{1-e^{-t-s}}{t+s}$. 
One can show (see e.g. \cite[Lemma~7]{Push})  that 
\begin{equation}
\sigma(\Gamma)=[0,\pi].
\label{b15}
\end{equation}
In fact, the spectrum of $\Gamma$ is purely absolutely continuous, 
but we will not need this fact. 
Identifying $L^2(\R_+,\calK)$ with $L^2(\R_+)\otimes\calK$, 
we denote the operators \eqref{b14} by $\Gamma\otimes F_0'(0)$
and $\Gamma\otimes F'(0)$. 

\subsection{The operators $L$ and $L_0$ }\label{sec.b5}

The crucial point of our proof of Theorem~\ref{thm.b3} is the  representation 
\begin{equation}
E(-1,0)E_0(0,1)=-LJL^*_0
\label{b16}
\end{equation}
in terms of some auxiliary operators $L_0$ and $L$
which we proceed to define. 
These operators act from $L^2(\R_+,\calK)$ 
to $\calH$. On the dense set $L^2(\R_+,\calK)\cap L^1(\R_+,\calK)$
we define $L_0$, $L$ by 
\begin{align}
L_0 f=\int_0^\infty e^{-tH_0}(GE_0(0,1))^*f(t)dt, 
\label{b17}
\\
L f=\int_0^\infty e^{tH}(GE(-1,0))^*f(t)dt.
\label{b18}
\end{align}

\begin{lemma}\label{lma.b6}
Assume \eqref{a3} and suppose that the derivatives
$F_0'(0)$, $F'(0)$ exist in the operator norm. 
Then:

(i) The operators $L_0$ and $L$ defined by \eqref{b17} and \eqref{b18}
extend to bounded operators from $L^2(\R_+,\calK)$ to $\calH$. 

(ii)
The differences 
\begin{equation}
L_0^*L_0-\Gamma\otimes F_0'(0),
\quad
L^* L-\Gamma\otimes F'(0)
\label{b19}
\end{equation}
are compact operators.

(iii) The identity \eqref{b16} holds true. 
\end{lemma}

\begin{proof}
(i) Let us prove that $L_0$ is bounded; the boundedness of $L$
is proven in the same way. 
For $f\in L^2(\R_+,\calK)\cap L^1(\R_+;\calK)$ we have
\[
\norm{L_0 f}^2
=
\int_0^\infty \int_0^\infty
(GE_0(0,1)e^{-(t+s)H_0}(GE_0(0,1))^*f(t),f(s))_{\calK}dt\,ds,
\]
and so the above expression is a quadratic form of the operator of the type 
\eqref{b13} with the kernel 
$K(t)=GE_0(0,1)e^{-tH_0}(GE_0(0,1))^*$.
By Proposition~\ref{lma.b5}, it suffices to prove the bound 
$\norm{K(t)}_\calK \leq C/t$, $t>0$.
Let $f\in\calH$ and $\rho(\lambda)=(E_0(-\infty,\lambda)f,f)$. 
Integrating by parts, one obtains
$$
\int_0^1 e^{-t\lambda}d\rho(\lambda)
=
e^{-t}\int_0^1d\rho(\lambda)
+
t\int_0^1 d\mu\,  e^{-t\mu}\int_0^\mu d\rho(\lambda).
$$
It follows that
$$
e^{-tH_0}E_0(0,1)
=
e^{-t}E_0(0,1)+t\int_0^1 e^{-t\mu}E_0(0,\mu)d\mu.
$$
Using this expression, the relation \eqref{a5a} and the fact that $GE_0(\{0\})=0$, we get
\begin{equation}
K(t)=e^{-t}(F_0(1)-F_0(0))+t\int_0^1 e^{-t\mu} (F_0(\mu)-F_0(0))d\mu.
\label{b24}
\end{equation}
By our assumption on the differentiability of $F_0$, we have
$$
\norm{F_0(\mu)-F_0(0)} \leq C\abs{\mu}
\text{ for $\abs{\mu}\leq1$}.
$$
Using this, we obtain:
\begin{multline*}
\norm{K(t)} \leq
e^{-t}\norm{F_0(1)-F_0(0)} 
+
t\int_0^1 e^{-t\mu}\norm{F_0(\mu)-F_0(0)} d\mu
\\
\leq
C e^{-t}+C t\int_0^1 e^{-t\mu}\mu d\mu
=
C(1-e^{-t})/t
\leq
C/t, \quad t>0,
\end{multline*}
as required. 

(ii) Let us consider the first of the differences \eqref{b19}; the second one
is considered in the same way. 
By the same reasoning as above, $L_0^*L_0-\Gamma\otimes F_0'(0)$
is the operator of the type \eqref{b13} with
$$
K(t)
=
GE_0(0,1)e^{-tH_0}(GE_0(0,1))^*
-
F_0'(0)
(1-e^{-t})/t.
$$
By \eqref{a3}, $F_0(\lambda)$ is compact for all $\lambda$. 
Since the derivative $F_0'(0)$ exists in the operator norm, 
the operator $F_0'(0)$ is also compact. 
Thus, $K(t)$ is compact for all $t>0$. 
By Proposition~\ref{lma.b5}(ii), it suffices to prove
that $\norm{K(t)}=o(1/t)$ as $t\to0$ and $t\to\infty$. 
For $t\to0$, the statement is obvious. Consider the 
limit $t\to\infty$. By the same calculation as in part (i) of 
the proof (see \eqref{b24}), we have
\begin{multline*}
K(t)=e^{-t}(F_0(1)-F_0(0))
+
t\int_0^1 e^{-t\mu}(F_0(\mu)-F_0(0))d\mu
\\
-F_0'(0)t \int_0^1 e^{-t\mu}\mu\  d\mu-F_0'(0)e^{-t}.
\end{multline*}
It follows that
\begin{equation}
\norm{K(t)} \leq e^{-t}\norm{F_0(1)-F_0(0)-F_0'(0)}
+
t\int_0^1 e^{-t\mu}\norm{F_0(\mu)-F_0(0)-F_0'(0)\mu} d\mu.
\label{b20}
\end{equation}
By our assumption, 
\begin{equation}
\norm{F_0(\mu)-F_0(0)-F_0'(0)\mu}=o(\mu) 
\text{ as $\mu\to0$.}
\label{b21}
\end{equation}
Using \eqref{b20} and \eqref{b21}, it is easy to see that
$\norm{K(t)}=o(1/t)$ as $t\to\infty$. 

(iii)
Let $f,f_0\in\calH$. 
Using \eqref{a4a}, we obtain
\begin{multline*}
\frac{d}{dt}(E_0(0,1)e^{-tH_0}f_0,E(-1,0)e^{tH}f)
\\
=
(E_0(0,1)e^{-tH_0}f_0,HE(-1,0)e^{tH}f)
-
(H_0E_0(0,1)e^{-tH_0}f_0,E(-1,0)e^{tH}f)
\\
=(JG E_0(0,1)e^{-tH_0}f_0,GE(-1,0)e^{tH}f)_\calK.
\end{multline*}
Using this and the easily verifiable relations
$$
\norm{E_0(0,1)e^{-tH_0}f_0}_\calH\to0,
\quad
\norm{E_0(-1,0)e^{tH}f}_\calH\to0
\quad \text{ as $t\to\infty$,}
$$
we get
\begin{multline*}
(JL_0^*f_0,L^*f)_{L^2(\R_+,\calK)}
=
\int_0^\infty (JG E_0(0,1)e^{-tH_0}f_0,GE(-1,0)e^{tH}f)_\calK dt
\\
=
\int_0^\infty \frac{d}{dt}(E_0(0,1)e^{-tH_0}f_0,E(-1,0)e^{tH}f) dt
=
-(E_0(0,1)f_0,E(-1,0)f),
\end{multline*}
which proves \eqref{b16}.
\end{proof}

\subsection{Proof of Theorem~\ref{thm.b3} }\label{sec.b6}
We will prove \eqref{b8}; the relation \eqref{b7} is 
proven in the same manner.

1. 
First we introduce some notation. 
For  bounded self-adjoint operators $M$ and $N$ we shall write 
\[
M\approx N \text{ if } M\mid_{(\Ker M)^\perp}
\text{ is unitarily equivalent to }N\mid_{(\Ker N)^\perp}.
\]
It is well known that $M^*M\approx MM^*$ for any bounded operator $M$;
below we use this fact.

2. 
Using Lemma~\ref{lma.b6} we get, for some compact operators
$X_0$ and $X$: 
$$
E_0(0,1)E(-1,0)E_0(0,1)
=
L_0JL^*LJL_0^*
=
L_0(\Gamma\otimes JF'(0)J)L_0^*+X,
$$
\begin{multline*}
L_0(\Gamma\otimes JF'(0)J)L_0^*
=
L_0(\Gamma^{1/2}\otimes JF'(0)^{1/2})(\Gamma^{1/2}\otimes F'(0)^{1/2}J)L_0^*
\\
\approx
(\Gamma^{1/2}\otimes F'(0)^{1/2}J)L_0^*L_0(\Gamma^{1/2}\otimes JF'(0)^{1/2})
\\
=
(\Gamma^{1/2}\otimes F'(0)^{1/2}J)(\Gamma\otimes F_0'(0))(\Gamma^{1/2}\otimes JF'(0)^{1/2})
+X_0
\\
=
\Gamma^2\otimes(F'(0)^{1/2}JF_0'(0)JF'(0)^{1/2})+X_0.
\end{multline*}
Thus, by Weyl's theorem, we obtain
$$
\sigma_\ess(E_0(0,1)E(-1,0)E_0(0,1))
=
\sigma_\ess(\Gamma^2\otimes Q), 
\quad 
Q=F'(0)^{1/2}JF_0'(0)JF'(0)^{1/2}.
$$

3. 
The operator $Q$ above is compact, selfadjoint and $Q\geq0$. 
Let $Q=\sum_{n=1}^\infty \lambda_n (\cdot, f_n)f_n$ be the 
spectral decomposition of $Q$, where $\lambda_1\geq \lambda_2\geq\cdots$ 
are the eigenvalues of $Q$. 
Then 
$$
\Gamma^2\otimes Q
=
\sum_{n=1}^\infty \lambda_n \Gamma^2\otimes(\cdot, f_n)f_n
$$ 
is an orthogonal sum decomposition of $\Gamma^2\otimes Q$, and therefore
$$
\sigma_\ess(\Gamma^2\otimes Q)
=
\cup_{n=1}^\infty 
\sigma_\ess(\lambda_n \Gamma^2\otimes (\cdot, f_n)f_n).
$$
Taking into account \eqref{b15} and recalling that 
$\lambda_1=\norm{Q}$, we obtain 
\begin{multline*}
\sigma_\ess(\Gamma^2\otimes Q)
=
\cup_{n=1}^\infty [0,\lambda_n\pi^2]=[0,\pi^2 \norm{Q}]
=
[0,\pi^2\norm{F'(0)^{1/2}JF_0'(0)JF'(0)^{1/2}}]
\\
=
[0,\pi^2\norm{F_0'(0)^{1/2}JF'(0)^{1/2}}^2]
=
[0,\alpha(0)^2],
\end{multline*}
as required. \qed

\section{Proofs of Theorems~\ref{cr.a2}, \ref{cr.a3}, \ref{cr.a4} and \ref{th5}}\label{sec.c}

\subsection{Existence of $F_0'$, $F'$ and $T_0$, $T$ }\label{sec.c1}
Here we recall various statements concerning the existence
of the derivatives $F_0'(\lambda)$, $F'(\lambda)$ and the limits
$T_0(\lambda+i0)$, $T(\lambda+i0)$ under the assumptions
of Theorems~\ref{cr.a2}, \ref{cr.a3}, \ref{cr.a4}. 
All of these statements are essentially well known. 
If the limits $T_0(\lambda+i0)$, $T(\lambda+i0)$ exist, we denote
$$
B_0(\lambda)=\Im T_0(\lambda+i0), 
\quad
B(\lambda)=\Im T(\lambda+i0).
$$
We first note that if the derivatives $F_0'(\lambda)$ and $F'(\lambda)$
and the limits $T_0(\lambda+i0)$, $T(\lambda+i0)$ exist at some point
$\lambda$, then 
\begin{equation}
\pi F_0'(\lambda)=B_0(\lambda),\quad 
\pi F'(\lambda)=B(\lambda). 
\label{c1}
\end{equation}
Indeed, this follows from the spectral theorem and the following 
well known fact  
(see e.g. \cite[Theorem~11.22]{Rudin}): if $\mu$ is a measure on $\R$
and the derivative  
$\frac{d}{d\lambda}\mu(-\infty,\lambda)$ exists, then 
$$
\pi\frac{d}{d\lambda}\mu(-\infty,\lambda)
=
\lim_{\varepsilon\to+0}\Im \int_\R \frac{d\mu(t)}{t-\lambda-i\varepsilon}.
$$

\begin{lemma}\label{prp.c1}
Assume \eqref{a3} and suppose that \eqref{a11} holds true 
for some bounded open interval $\Delta\subset\R$. Then for all 
$\lambda\in\Delta$ the derivatives $F_0'(\lambda)$, $F'(\lambda)$
exist in the operator norm.
\end{lemma}
\begin{proof}
From the obvious operator inequality 
$$
0\leq E_0(\{\lambda\})\leq \frac{\varepsilon^2}{(H_0-\lambda)^2+\varepsilon^2},
\quad \varepsilon>0,
$$
we get 
$$
0\leq GE_0(\{\lambda\})(GE_0(\{\lambda\}))^*
\leq
\varepsilon \Im T_0(\lambda+i\varepsilon), 
\quad \varepsilon>0.
$$
This implies that $GE_0(\{\lambda\})=0$ for all $\lambda\in\Delta$. 
Using this, Stone's formula (see e.g. \cite[Theorem~VII.13]{RS1}) yields 
$$
((F_0(b)-F_0(a))f,f)
=
\lim_{\varepsilon\to+0}\frac1\pi\int_a^b
\Im(T_0(\lambda+i\varepsilon)f,f)d\lambda
=
\frac1\pi\int_a^b (B_0(\lambda)f,f)d\lambda
$$
for any interval $[a,b]\subset\Delta$.
From here and the continuity of $B_0(\lambda)$
we get the statement concerning $F_0'(\lambda)$. 
The case of $F'(\lambda)$ is considered in the same way. 
\end{proof}

\begin{lemma}\label{prp.c2}
(i) Assume that $G$ is a Hilbert-Schmidt operator. 
Then for a.e. $\lambda\in\R$, the derivatives $F_0'(\lambda)$,
$F'(\lambda)$ and the limits $T_0(\lambda+i0)$, $T(\lambda+i0)$ 
exist in the operator norm. 

(ii) Under the assumptions of Theorem~\ref{cr.a4}, for a.e.
$\lambda\in\R$ the derivatives $F_0'(\lambda)$, $F'(\lambda)$
and the limits $T_0(\lambda+i0)$, $T(\lambda+i0)$ exist in the 
operator norm.
\end{lemma}
\begin{proof}
(i) is one of the key facts of the trace class scattering 
theory,  see e.g. \cite[Section~6.1]{Yafaev}.

(ii) 
First consider $F_0'$ and $T_0$.
Let us apply a standard argument: let $\Delta_1=(-R,R)$, 
$\Delta_2=\R\setminus\Delta_1$ and write
$G_j=GE_0(\Delta_j)$, $j=1,2$. 
Then $G_1\in\mathbf S_2$. 
Thus, by part (i) of the lemma, the derivative
$$
\frac{d}{d\lambda}F_0(\lambda)
=
\frac{d}{d\lambda}G_1E_0(-\infty,\lambda)G_1^*, 
\quad \lambda\in\Delta_1,
$$
exists in the operator norm.
Let us consider $T_0(z)$; we have
\begin{equation}
T_0(z)
=
G_1R_0(z)G_1^*
+
G_2(G_2R_0(\overline z))^*.
\label{c17}
\end{equation}
By part (i) of the lemma, the first term in the r.h.s. of \eqref{c17} has a limit
as $z\to\lambda+i0$ for a.e. $\lambda\in\Delta_1$.
Since $R_0(z)E_0(\Delta_2)$ is analytic in $z\in\C\setminus\overline{\Delta_2}$, 
the second term in the r.h.s. of \eqref{c17} has a limit as $z\to\lambda+i0$ 
for all $\lambda\in\Delta_1$. 
It follows that $T_0(z)$ has boundary values 
as $z\to\lambda+i0$ for a.e. $\lambda\in\Delta_1$.
Since $R$ in the definition of $\Delta_1$ can be taken arbitrary large, 
this gives the desired 
statement for a.e. $\lambda\in\R$. 

Consider $F'$ and $T$. First, exactly as in the proof of 
\cite[Theorem~XI.30]{RS3}, using \eqref{a3} and \eqref{a13}, 
one shows that 
\begin{equation}
G(\abs{H}+I)^{-m}\in\mathbf S_2.
\label{c2}
\end{equation}
After this, the proof follows the same argument
as above. 
\end{proof}

\begin{lemma}\label{lma.c5}
Assume \eqref{a3} and let $\Delta\subset\R$ 
be a bounded interval. 
Suppose that for a.e. $\lambda\in\Delta$, the 
derivative $F_0'(\lambda)$ exists in the operator 
norm. Then for a.e. $\lambda\in\Delta\setminus \hat\sigma_\ac(H_0)$, 
one has $F_0'(\lambda)=0$. 
\end{lemma}
\begin{proof}
1. Recall the following measure theoretic statement. Let $\mu$ be a finite
Borel measure on $\R$ and let $Z$ be a Borel support of $\mu$, 
i.e. $\mu(\R\setminus Z)=0$. 
Then 
\begin{equation}
\frac{d}{d\lambda}\mu(-\infty,\lambda)=0
\quad\text{ for Lebesgue-a.e. $\lambda\in\R\setminus Z$.}
\label{c19}
\end{equation}
Indeed, let $\mu=\mu_\ac+\mu_s$ be the decomposition 
of $\mu$ into the a.c. and singular components with respect
to the Lebesgue measure. Let 
$0\leq f\in L^1(\R)$ be the Radon-Nikodym 
derivative of  $\mu_\ac$ with respect to the 
Lebesgue measure. Then (see e.g. \cite[Section~8.6]{Rudin})
$$
\frac{d}{d\lambda}\mu_\ac(-\infty,\lambda)=f(\lambda), 
\quad 
\frac{d}{d\lambda}\mu_s(-\infty,\lambda)=0, 
\quad\text{ for Lebesgue-a.e. $\lambda\in\R$.}
$$
The statement $\mu(\R\setminus Z)=0$ implies that
$$
\int_{\R\setminus Z} f(\lambda)d\lambda=0.
$$
Thus, $f(\lambda)=0$ for Lebesgue-a.e. $\lambda\in\R\setminus Z$. 
From here we get \eqref{c19}.

2. Let $Z_s$ be a Borel support of the singular part of 
the spectral measure $E_0$. Since the Lebesgue measure
of $Z_s$ is zero, the set $\hat \sigma=\hat \sigma_\ac(H_0)\cup Z_s$
is again a core of the a.c. spectrum of $H_0$. Moreover, 
$\hat\sigma$ is a Borel support of $E_0$, i.e. 
$E_0(\R\setminus \hat\sigma)=0$.

3. Let $G_\Delta=GE_0(\Delta)$; by \eqref{a3}, $G_\Delta$ is a compact operator. 
Let $\{e_n\}_{n=1}^\infty$ be an orthonormal basis in $\calK$. 
Consider the complex valued measures
$$
\mu_{nm}(\Lambda)=(E_0(\Lambda)G_\Delta^*e_n,G_\Delta^*e_m),
\quad n,m\in\N, \quad \Lambda\subset \Delta.
$$
We have $\mu_{nm}(\Delta\setminus\hat \sigma)=0$. 
Representing each $\mu_{nm}$ as a linear combination of four non-negative
measures and applying \eqref{c19}, we obtain
$$
\frac{d}{d\lambda}\mu_{nm}(-\infty,\lambda)=0,
\quad \lambda\in (\Delta\setminus\hat \sigma)\setminus \Lambda_{nm},
\quad n,m\in\N,
$$
where the Lebesgue measure of the set $\Lambda_{nm}$ is zero.
It follows that 
\begin{equation}
\frac{d}{d\lambda}\mu_{nm}((-\infty,\lambda))=0,
\quad \lambda\in (\Delta\setminus\hat \sigma)\setminus \Lambda,
\quad n,m\in\N,
\label{c22}
\end{equation}
where $\Lambda=\cup_{n,m}\Lambda_{nm}$ and the Lebesgue measure of $\Lambda$ is zero.

4. Let $\mathcal D\subset\calK$ be the dense set of all finite linear combinations
of elements of the basis $\{e_n\}_{n=1}^\infty$. It follows from \eqref{c22} that
$$
\frac{d}{d\lambda}(E_0(-\infty,\lambda)G_\Delta^*f,G_\Delta^*g)=0,
\quad \forall f,g\in\mathcal D, 
\quad \text{ a.e. $\lambda\in\Delta\setminus\hat\sigma$,}
$$
and therefore $F_0'(\lambda)=0$ for a.e. $\lambda\in\Delta\setminus\hat\sigma$. 
\end{proof}

\subsection{Connection between $\alpha(\lambda)$ and $S(\lambda)$}\label{sec.c2}
First we establish a connection between $\alpha(\lambda)$ and some 
auxiliary unitary operator $\wt S(\lambda)$. The idea to use 
the operator $\wt S(\lambda)$ is due to A.~V.~Sobolev and 
D.~R.~Yafaev \cite{SobYaf}.

\begin{lemma}\label{lma.c8}
Assume \eqref{a3} and suppose that 
the derivatives $F_0'(\lambda)$ and $F'(\lambda)$ 
and the limits
$T_0(\lambda+i0)$, $T(\lambda+i0)$ exist 
for some $\lambda\in\R$. 
Then 
the operator
\begin{equation}
\wt S(\lambda)
=
I-2 i B_0(\lambda)^{1/2}(J-JT(\lambda+i0)J)B_0(\lambda)^{1/2}
\label{c3}
\end{equation}
in $\calK$ is unitary and 
\begin{equation}
\frac12\norm{\wt S(\lambda)-I}=\alpha(\lambda). 
\label{c3a}
\end{equation}
\end{lemma}
\begin{proof}
1. 
From \eqref{a21} one easily obtains the identity
\begin{equation}
I-T(z)J
=
(I+T_0(z)J)^{-1}, 
\quad \Im z>0.
\label{c7}
\end{equation}
Since the limits $T_0(\lambda+i0)$ and $T(\lambda+i0)$ 
exist in the operator norm, we conclude that
the operator $I+T_0(\lambda+i0)J$ has a
bounded inverse and 
\begin{equation}
I-T(\lambda+i0)J
=
(I+T_0(\lambda+i0)J)^{-1}.
\label{c7a}
\end{equation}
In the same way, one obtains
\begin{equation}
I-JT(\lambda+i0)
=
(I+JT_0(\lambda+i0))^{-1}.
\label{c7b}
\end{equation}
Taking adjoints in \eqref{c7b} and subtracting from \eqref{c7a},
after some simple algebra we get
$$
JB(\lambda)J
=
(J-JT(\lambda+i0)J)B_0(\lambda)(J-JT(\lambda+i0)^*J).
$$
From here the unitarity of $\wt S(\lambda)$ follows 
by a direct calculation. 

2. 
Using the unitarity of $\wt S(\lambda)$ 
and the identity \eqref{c1}, we obtain 
\begin{multline*}
(\wt S(\lambda)-I)^*(\wt S(\lambda)-I)
=
2I-2\Re\wt S(\lambda)
=
4\ \Im(B_0(\lambda)^{1/2}JT(\lambda+i0)JB_0(\lambda)^{1/2})
\\
=
4\ \Im(B_0(\lambda)^{1/2}JB(\lambda)JB_0(\lambda)^{1/2})
=
4\pi^2\Im(F_0'(\lambda)^{1/2}JF'(\lambda)JF_0'(\lambda)^{1/2}).
\end{multline*}
From here, taking into account \eqref{a7},  we get
$$
\frac14\norm{\wt S(\lambda)-I}^2
=
\pi^2
\norm{F_0'(\lambda)^{1/2}JF'(\lambda) JF_0'(\lambda)^{1/2}}
=
\pi^2\norm{F_0'(\lambda)^{1/2}JF'(\lambda)^{1/2}}^2
=
\alpha(\lambda)^2,
$$
as required. 
\end{proof}
The following Lemma is essentially contained in 
\cite[Section~7.7]{Yafaev}.

\begin{lemma}\label{prp.c3}
(i) Under the assumptions \eqref{a3}, \eqref{a11}, 
the local wave operators 
$W_\pm(H_0,H;\Delta)$ exist and are complete, and
for a.e. $\lambda\in\hat\sigma_\ac(H_0)\cap\Delta$ we have
\begin{equation}
\norm{S(\lambda)-I}=\norm{\wt S(\lambda)-I}.
\label{c4}
\end{equation}

(ii) Under the assumptions of Theorem~\ref{cr.a3} or \ref{cr.a4},
the wave operators $W_\pm(H_0,H)$ exist and are complete, and
for a.e. $\lambda\in\hat\sigma_\ac(H_0)$, the relation
\eqref{c4} holds true.
\end{lemma}
\begin{proof}
(i) For the existence and completeness of wave operators, 
we refer to \cite[Section~4.5]{Yafaev}. 
Next, for a.e. 
$\lambda\in\hat\sigma_\ac(H_0)\cap\Delta$, 
the scattering matrix can be represented as 
\begin{equation}
S(\lambda)
=
I-2\pi iZ(\lambda)(J-JT(\lambda+i0)J)Z(\lambda)^*, 
\label{c5}
\end{equation}
where $Z(\lambda):\calK\to\mathfrak h(\lambda)$ is an 
operator such that
\begin{equation}
\pi Z(\lambda)^*Z(\lambda)=B_0(\lambda).
\label{c6}
\end{equation}
This is the well known stationary representation for the 
scattering matrix, see e.g. \cite[Section~5.5(3)]{Yafaev}. 
Let us use the polar decomposition of $Z(\lambda)$, 
$Z(\lambda)=U\abs{Z(\lambda)}$, where
$\abs{Z(\lambda)}=\sqrt{Z(\lambda)^*Z(\lambda)}=B_0(\lambda)^{1/2}/\pi$,
and $U$ is an isometry which maps $\overline{\Ran Z(\lambda)^*}$
onto $\overline{\Ran Z(\lambda)}$. 
Then we get
$$
S(\lambda)-I=U(\wt S(\lambda)-I)U^*,
$$
and \eqref{c4} follows. 
This argument is borrowed from \cite[Lemma~7.7.1]{Yafaev}.

(ii)
Existence and completeness of wave operators
is well known, see e.g. 
\cite[Theorem~6.4.5]{Yafaev}. 
As in the proof of part (i), we have the representation
\eqref{c5}, \eqref{c6} for a.e. $\lambda\in\hat\sigma_\ac(H_0)$
(see e.g. \cite[Section~5.5(3)]{Yafaev})
and the required statement follows by the same argument as above. 
\end{proof}

\begin{proof}[Proof of Theorem~\ref{cr.a2}] 
The existence of the derivatives $F_0'(\lambda)$ 
and $F'(\lambda)$ follows from Lemma~\ref{prp.c1}. 
Thus, by Theorem~\ref{th1}, we obtain \eqref{a8}. 
By Lemma~\ref{lma.c8} and Lemma~\ref{prp.c3}, we have
$$
\alpha(\lambda)=\frac12 
\norm{\wt S(\lambda)-I}=\frac12 \norm{S(\lambda)-I}
$$
for a.e. $\lambda\in\hat\sigma_\ac(H_0)\cap\Delta$.
Thus, we have \eqref{a9}  and therefore \eqref{a1} 
for a.e. $\lambda\in\hat\sigma_\ac(H_0)\cap\Delta$. 
On the other hand, for a.e. $\lambda\in\Delta\setminus\hat\sigma_\ac(H_0)$, by 
Lemma~\ref{lma.c5}, we have $\alpha(\lambda)=0$.
Thus, according to \eqref{a23}, the relations \eqref{a9} and \eqref{a1}
hold true also for a.e. $\lambda\in\Delta\setminus\hat\sigma_\ac(H_0)$.
\end{proof}

\begin{proof}[Proof of Theorems~\ref{cr.a3} and \ref{cr.a4}]
By Lemma~\ref{prp.c2}, the derivatives $F_0'(\lambda)$, 
$F'(\lambda)$ and the limits $T_0(\lambda+i0)$, $T(\lambda+i0)$
exist for a.e. $\lambda\in\R$. 
Thus, the identity  \eqref{a8} follows from Theorem~\ref{th1}.
The identities \eqref{a9} and \eqref{a1} follow for a.e. 
$\lambda\in\R$ 
as in the proof of  Theorem~\ref{cr.a2}. 
\end{proof}

\subsection{The Fredholm property}\label{sec.c3}
\begin{proof}[Proof of Theorem~\ref{th5}]

1. As in the proof of Lemma~\ref{lma.c8}, 
we get that the operators $I+T_0(\lambda+i0)J$ and
$I+JT_0(\lambda+i0)$ have bounded inverses and the 
identities \eqref{c7a}, \eqref{c7b} hold true.

2. 
From \eqref{c7a}, \eqref{c7b} we obtain
$$
I-A(\lambda)J
=
(I+T_0(\lambda+i0)J)^{-1}(I+A_0(\lambda)J)(I+T_0(\lambda+i0)^*J)^{-1}.
$$
This proves that $\dim\Ker(I-A(\lambda)J)=\dim\Ker(I+A_0(\lambda)J)$ 
and so (ii)$\Leftrightarrow$(iii).

3. 
Let us prove that 
\begin{equation}
\dim\Ker(I+\wt S(\lambda))=
\dim\Ker(I+A_0(\lambda)J).
\label{c9}
\end{equation}
Using the identity \eqref{c7a} and the fact that 
$\dim\Ker(I+XY)=\dim\Ker(I+YX)$ for any bounded 
operators $X$ and $Y$, we obtain:
\begin{align*}
\dim\Ker(I+\wt S(\lambda))
&=
\dim\Ker(I-iB_0(\lambda)^{1/2}J(I-T(\lambda+i0)J) B_0(\lambda)^{1/2})
\\
&=
\dim\Ker(I-iB_0(\lambda)^{1/2}J(I+T_0(\lambda+i0)J)^{-1} B_0(\lambda)^{1/2})
\\
&=\dim\Ker(I-iB_0(\lambda)J(I+T_0(\lambda+i0)J)^{-1})
\\
&=\dim\Ker(I+T_0(\lambda+i0)J-iB_0(\lambda)J)
\\
&=\dim\Ker(I+A_0(\lambda)J),
\end{align*}
as required.

4.
Let us prove that (i)$\Leftrightarrow$(ii). 
By the definition \eqref{a13a} and by Theorem~\ref{th1}, it suffices to prove that 
$\alpha(\lambda)<1$ if and only if $\Ker(I+A_0(\lambda)J)=\{0\}$. 
Suppose that $\Ker(I+A_0(\lambda)J)=\{0\}$. 
Then by \eqref{c9}, we have $\Ker(I+\wt S(\lambda))=\{0\}$. 
Since $\wt S(\lambda)-I$ is compact,
it follows that $-1\notin\sigma(\wt S(\lambda))$. 
Since $\wt S(\lambda)$ is unitary, we get $\norm{\wt S(\lambda)-I}<2$. 
By \eqref{c3a}, it follows that $\alpha(\lambda)<1$.
 
Conversely, suppose that $\dim\Ker(I+A_0(\lambda)J)>0$. 
Then $\dim\Ker(I+\wt S(\lambda))>0$ and therefore 
$\norm{\wt S(\lambda)-I}=2$. 
By \eqref{c3a}, it follows that $\alpha(\lambda)=1$. 
\end{proof}

\section{Piecewise continuous functions $\varphi$}\label{sec.d}

We closely follow the proof used by S.~Power in his description 
\cite{Power}
of 
the essential spectrum of Hankel operators with piecewise 
continuous symbols.
We use the shorthand notation 
$$
\delta(\varphi)=\varphi(H)-\varphi(H_0).
$$
\subsection{Auxiliary statements }\label{sec.d1}

\begin{lemma}\label{lma.c4}
Assume \eqref{a3} and let $\varphi_1,\varphi_2\in PC_0(\R)$. 
Suppose that $$
\singsupp \varphi_1\cap\singsupp \varphi_2=\varnothing.
$$
Then the operator $\delta(\varphi_1)\delta(\varphi_2)$ is compact. 
\end{lemma}
\begin{proof}
1. 
For $j=1,2$ one can represent $\varphi_j$ as 
$\varphi_j=\psi_j+\zeta_j$, 
where $\zeta_j\in C_0(\R)$, $\psi_j\in PC_0(\R)$  and 
$\supp \psi_1\cap\supp \psi_2=\varnothing$. 
By Lemma~\ref{lma.b4}, the operators 
$\delta(\zeta_1)$ and $\delta(\zeta_2)$ are compact. 
We have
$$
\delta(\varphi_1)\delta(\varphi_2)
=
(\delta(\psi_1)+\delta(\zeta_1))(\delta(\psi_2)+\delta(\zeta_2))
$$
and so it suffices to prove that the operator 
$\delta(\psi_1)\delta(\psi_2)$ is compact. 

2. One can choose $\omega\in C_0(\R)$ 
such that $\psi_1\omega=\psi_1$ and $\omega\psi_2\equiv0$. 
Then $\omega(H_0)\psi_2(H_0)=0$ and 
$$
\psi_1(H)\psi_2(H_0)
=
\psi_1(H)\omega(H)\psi_2(H_0)
\\
=
\psi_1(H)(\omega(H)-\omega(H_0))\psi_2(H_0),
$$
and the operator in the r.h.s. is compact by Lemma~\ref{lma.b4}.
By the same argument, the operator $\psi_1(H_0)\psi_2(H)$
is compact. It follows that the operator 
\begin{multline*}
\delta(\psi_1)\delta(\psi_2)
=
(\psi_1(H)-\psi_1(H_0))(\psi_2(H)-\psi_2(H_0))
\\
=
-\psi_1(H_0)\psi_2(H)
-\psi_1(H)\psi_2(H_0)
\end{multline*}
is compact, as required. 
\end{proof}

\begin{lemma}\label{lma.c6}
Let $A_n$, $n=1,\dots,N$,  be bounded operators in a Hilbert
space. Assume that 
$A_nA_m$ is compact for all $n\not=m$.  Then 
\begin{equation}
\sigma_\ess(A_1+\dots+A_N)\cup\{0\}
=
\left(\cup_{j=1}^N\sigma_\ess(A_j)\right).
\label{c13}
\end{equation}
\end{lemma}
See e.g. \cite[Section~10.1]{Peller} for a proof via the 
Calkin algebra argument. 
We would like to emphasise that Lemma~\ref{lma.c6} 
holds true with the definition of the essential spectrum 
as stated in Section~\ref{sec.a6}; it is in general false for some
other definitions of the essential spectrum, see e.g. 
\cite[Section~XIII.4, Example~1]{RS4}.

\subsection{Proof of Theorem~\ref{th10} }\label{sec.d2}
We start by considering the case of finitely many discontinuities:

\begin{lemma}\label{lma.c7}
Assume the hypothesis of Theorem~\ref{th10} and suppose in 
addition that the set $\singsupp \varphi$ is finite. 
Then the conclusion of Theorem~\ref{th10}  holds true.
\end{lemma}
\begin{proof}
1. 
First assume that $\varphi$ has only one discontinuity, i.e.
$\singsupp \varphi=\{\lambda_0\}$. 
Denote
\begin{equation}
\wt \varphi(\lambda)=
(\varphi(\lambda_0+0)-\varphi(\lambda))/
\varkappa_{\lambda_0}(\varphi).
\label{c13a}
\end{equation}
Then $\wt \varphi(\lambda_0-0)=1$, $\wt\varphi(\lambda_0+0)=0$. 
We can write $\wt\varphi$ as
$$
\wt \varphi(\lambda)
=
\chi_{(-\infty,\lambda_0)}(\lambda)+\zeta(\lambda),
$$
where $\chi_{(-\infty,\lambda_0)}(\lambda)$ is the characteristic
function of $(-\infty,\lambda_0)$ and $\zeta\in C(\R)$ is such that
the limits of $\zeta(\lambda)$ as $\lambda\to\pm\infty$ exist. 
Then by Lemma~\ref{lma.b4}, $\zeta(H)-\zeta(H_0)$ is compact, 
and so 
$$
\wt\varphi(H)-\wt\varphi(H_0)
=
E(-\infty,\lambda_0)-E_0(-\infty,\lambda_0)+\text{ compact operator}.
$$
By Theorem~\ref{cr.a2} and Weyl's theorem on the invariance of the essential 
spectrum under the compact perturbations, we obtain
$$
\sigma_\ess(\wt\varphi(H)-\wt\varphi(H_0))=[-\alpha(\lambda_0),\alpha(\lambda_0)].
$$
Recalling the definition \eqref{c13a} of $\varphi$, we obtain 
\begin{equation}
\sigma_\ess(\varphi(H)-\varphi(H_0))
=
[-\alpha(\lambda_0)\varkappa_{\lambda_0}(\varphi),
\alpha(\lambda_0)\varkappa_{\lambda_0}(\varphi)].
\label{c18}
\end{equation}

2. Consider the general case; let 
$\singsupp\varphi=\{\lambda_1,\dots,\lambda_N\}\subset\Delta$. 
One can represent $\varphi=\sum_{n=1}^N \varphi_n$, 
where $\varphi_n\in PC_0(\R)$, $\singsupp\varphi_n=\{\lambda_n\}$
and $\varkappa_{\lambda_n}(\varphi_n)=\varkappa_{\lambda_n}(\varphi)$
for each $n$. 
Then 
$$
\delta(\varphi)
=
\sum_{n=1}^N\delta(\varphi_n),
$$
and by Lemma~\ref{lma.c4}, the operators 
$\delta(\varphi_n)\delta(\varphi_m)$
are compact for $n\not=m$. 
Applying Lemma~\ref{lma.c6} and the first step of the proof, we get
$$
\sigma_\ess(\delta(\varphi))\cup\{0\}
=
\cup_{n=1}^N\sigma_\ess(\delta(\varphi_n))
=\cup_{n=1}^N
[-\alpha(\lambda_n)\varkappa_{\lambda_n}(\varphi),
\alpha(\lambda_n)\varkappa_{\lambda_n}(\varphi)].
$$
Since $\sigma_\ess$ is a closed set, we get $0\in\sigma_\ess(\delta(\varphi))$
and thus the required statement \eqref{a16} follows. 
\end{proof}

\begin{proof}[Proof of Theorem~\ref{th10}]
1. 
Let 
\begin{align*}
\Lambda_0&=\{\lambda\in\Delta \mid \abs{\varkappa_\lambda(\varphi)}\geq1\},
\\
\Lambda_n&=\{\lambda\in\Delta \mid 2^{-n}\leq\abs{\varkappa_\lambda(\varphi)}<2^{-n+1}\},
\quad
n=1,2,\dots.
\end{align*}
The set $\Lambda_n$ is finite for all $n\geq0$. 
It is easy to see that for each $n\geq0$ there exists a function 
$\varphi_n\in PC_0(\R)$ with $\singsupp \varphi_n=\Lambda_n$, 
$\supp \varphi_n\subset\Delta$, and 
\begin{gather}
\varkappa_{\lambda}(\varphi_n)=\varkappa_\lambda(\varphi)
\quad \forall \lambda\in\Lambda_n,
\notag
\\
\norm{\varphi_n}_\infty=\frac12\max_{\lambda\in\Lambda_n}
\abs{\varkappa_\lambda(\varphi)}\leq 2^{-n}.
\label{c24}
\end{gather}
With this choice, the series $\sum_{n\geq0} \varphi_n$
converges absolutely and uniformly on $\R$ and defines a 
function $f=\sum_{n\geq0}\varphi_n$
such that $f\in PC_0(\R)$ and the function 
$\zeta\overset{\rm def}{=}\varphi-f$ is in the class $C_0(\R)$. 

2. 
For a given $N\in\mathbb N$, write 
$$
\varphi=f_N+g_N+\zeta, 
\quad 
f_N=\sum_{n=0}^N \varphi_n, 
\quad 
g_N=\sum_{n=N+1}^\infty \varphi_n. 
$$
By Lemma~\ref{lma.c4}, the operator $\delta(\varphi_m)\delta(\varphi_n)$ is compact
for $n\not=m$. 
By the estimate \eqref{c24},  the series in the r.h.s. of 
$$
\delta(\varphi_m)\delta(g_N)
=
\sum_{n=N+1}^\infty \delta(\varphi_m)\delta(\varphi_n)
$$
converges in the operator norm, and so for any $m\leq N$
the operator $\delta(\varphi_m)\delta(g_N)$
is also compact. 
Applying Lemma~\ref{lma.c6} to the decomposition 
$\delta(\varphi)=\delta(f_N)+\delta(g_N)+\delta(\zeta)$ 
and subsequently using Lemma~\ref{lma.c7}, we get
\begin{align*}
\sigma_\ess(\delta(\varphi))\cup\{0\}
&=
\sigma_\ess(\delta(f_N))\cup\sigma_\ess(\delta(g_N))\cup\{0\}
\\
&=
\biggl(\bigcup_{n=0}^N \sigma_\ess(\delta(\varphi_n))\biggr)
\cup
\sigma_\ess(\delta(g_N))\cup\{0\}
\\
&=
\biggl(\bigcup_{\abs{\varkappa_\lambda(\varphi)}\leq 2^{1-N}}
[-\alpha(\lambda)\varkappa_\lambda(\varphi),
\alpha(\lambda)\varkappa_\lambda(\varphi)]\biggr)
\cup
\sigma_\ess(\delta(g_N)).
\end{align*}
Finally, by the estimate \eqref{c24}
we have $\norm{\delta(g_N)}\leq 2\norm{g_N}_\infty\leq 2^{1-N}$
and so $\sigma_\ess(\delta(g_N))\subset\{z\in\C\mid \abs{z}\leq 2^{1-N}\}$.
Since $N$ can be taken arbitrary large, we obtain
$$
\sigma_\ess(\delta(\varphi))\cup\{0\}
=
\cup_{\lambda\in\Delta}
[-\alpha(\lambda)\varkappa_\lambda(\varphi),
\alpha(\lambda)\varkappa_\lambda(\varphi)].
$$
Since $\sigma_\ess$ is a closed set, we get $0\in\sigma_\ess(\delta(\varphi))$
and thus the required statement \eqref{a16} follows. 
\end{proof}

\section*{Acknowledgements}
The author is grateful to Nikolai Filonov, Serge Richard and Dmitri Yafaev  
for the critical reading of the manuscript and making  a number of useful 
suggestions.


\begin{thebibliography}{10}

\bibitem{Peller3}
{\sc A.~Aleksandrov, V.~Peller,}
\emph{Operator H\"older--Zygmund functions,}
preprint 2009.


\bibitem{ASS}
{\sc J.~Avron,  R.~Seiler,  B.~Simon,}
\emph{The index of a pair of projections,} 
J. Funct. Anal. \textbf{120} (1994), no. 1, 220--237. 


\bibitem{BS1}
{\sc M.~Sh.~Birman, M.~Z.~Solomyak,} 
\emph{Double Stieltjes operator integrals. III,}
(in Russian),
Problems of Math. Phys., Leningrad. 
Univ. \textbf{6} (1973), 27--53 


\bibitem{BS2}
{\sc M.~Sh.~Birman, M.~Z.~Solomyak,} 
\emph{Double Operator Integrals in a Hilbert Space,}
Integr. Equ. Oper. Theory \textbf{47} (2003), 131--168.


\bibitem{EE}
{\sc D.~E.~Edmunds and W.~D.~Evans,}
\emph{Spectral theory and differential operators,}
Oxford: Clarendon Press, 1987. 



\bibitem{Halmos}
{\sc P.~R.~Halmos,  }
\emph{Two subspaces,} 
Trans. Amer. Math. Soc. \textbf{144} (1969), 381--389.





\bibitem{Howland3}
{\sc J.~S.~Howland, }
\emph{Spectral theory of operators of Hankel type. II} 
Indiana Univ. Math. J. \textbf{41} (1992), no. 2, 427--434.


\bibitem{Kato2}
{\sc T.~Kato, }
\emph{Wave operators and similarity for some non-selfadjoint operators,}
Math. Annalen  \textbf{162} (1966), 258--279.


\bibitem{KM}
{\sc V.~Kostrykin, K.~Makarov, }
\emph{On Krein's example,}
Proc. Amer. Math. Soc. \textbf{136} (2008),  no. 6, 2067--2071.


\bibitem{Krein}
{\sc M.~G.~Krein, }
\emph{On the trace formula in perturbation theory,}
(in Russian)
Mat. Sb. \textbf{33 (75)} (1953), no. 3, 597--626.

\bibitem{Peller4}
{\sc F.~Nazarov, V.~Peller,}
\emph{Lipschitz functions of perturbed operators,}
preprint 2009.



\bibitem{Peller1}
{\sc V.~V.~Peller,} 
\emph{Hankel operators in perturbation theory of unitary and self-adjoint 
operators,} Funct. Anal. Appl. \textbf{19} (1985), 111--123. 


\bibitem{Peller2}
{\sc V.~V.~Peller,} 
\emph{Hankel operators in perturbation theory of unbounded self-adjoint operators,} 
Analysis and partial differential equations, 529Ð544, Lecture Notes in Pure and Appl. 
Math., 122, Dekker, New York, 1990.


\bibitem{Peller}
{\sc V.~V.~Peller,}
\emph{Hankel operators and their applications,}
Springer, 2003


\bibitem{Power}
{\sc S.~Power,}
\emph{Hankel operators on Hilbert space,}
Pitman, London, 1982.



\bibitem{Push2}
{\sc A.~Pushnitski,}
\emph{The spectral shift function and the invariance principle,}
J. Functional Analysis, \textbf{183},    (2001), no.~2,  269--320.


\bibitem{Push}
{\sc A.~Pushnitski,}
\emph{The scattering matrix and the differences of spectral projections},
Bulletin London Math. Soc. \textbf{40}  (2008), 227--238.



\bibitem{Push3}
{\sc A.~Pushnitski,}
\emph{The Birman-Schwinger principle on the continuous spectrum,}
in preparation. 

\bibitem{PYaf}
{\sc A.~Pushnitski, D.~Yafaev,}
\emph{Spectral theory of 
discontinuous functions of self-adjoint operators  and scattering theory},
preprint 2009.


\bibitem{RS1}
{\sc M.~Reed, B.~Simon,}
\emph{Methods of modern mathematical physics. I. Functional Analysis.}
Academic Press, New York -- London, 1972.


\bibitem{RS2}
{\sc M.~Reed, B.~Simon,}
\emph{Methods of modern mathematical physics. II. Fourier Analysis, Self-adjointness.}
Academic Press, New York -- London, 1975.



\bibitem{RS3}
{\sc M.~Reed, B.~Simon,}
\emph{Methods of modern mathematical physics. III. Scattering theory.}
Academic Press, New York -- London, 1979.



\bibitem{RS4}
{\sc M.~Reed, B.~Simon},  
\emph{Methods of Modern
Masthematical Physics. IV. Analysis of Operators}, 
Academic Press, New York -- London, 1978.

\bibitem{Rudin}
{\sc W.~Rudin,}
\emph{Real and complex analysis,}
McGraw-Hill, 1970.


\bibitem{Simon}
{\sc B.~Simon, }
\emph{Schr\"odinger semigroups}, 
Bulletin AMS \textbf{7} (1982), no.~3, 447--526.


\bibitem{SobYaf}
{\sc A.~V.~Sobolev, D.~R.~Yafaev, }
\emph{Spectral properties of an abstract scattering matrix}, 
Proc. Steklov Inst. Math.  \textbf{3} (1991), 159--189.

\bibitem{Yafaev}
{\sc  D.~R.~Yafaev,}
\emph{Mathematical scattering theory. General theory.}
American Mathematical Society, Providence, RI, 1992. 



\end{thebibliography}
\end{document}